\documentclass[10pt, reqno]{amsart}
\usepackage{amsmath, amsthm, amscd, amsfonts, amssymb, graphicx, color}
\usepackage[bookmarksnumbered, colorlinks, plainpages]{hyperref}
\hypersetup{colorlinks=true,linkcolor=red, anchorcolor=green, citecolor=cyan, urlcolor=red, filecolor=magenta, pdftoolbar=true}
\usepackage{mathrsfs}

\usepackage[OT2,OT1]{fontenc}
\newcommand\cyr
{
\renewcommand\rmdefault{wncyr}
\renewcommand\sfdefault{wncyss}
\renewcommand\encodingdefault{OT2}
\normalfont
\selectfont
}
\DeclareTextFontCommand{\textcyr}{\cyr}

\newtheorem{theorem}{Theorem}[section]
\newtheorem{lemma}[theorem]{Lemma}

\newtheorem{corollary}[theorem]{Corollary}
\theoremstyle{definition}
\newtheorem{definition}[theorem]{Definition}

\theoremstyle{remark}
\newtheorem{remark}[theorem]{Remark}
\numberwithin{equation}{section}

\begin{document}
\setcounter{page}{1}

\title[Arithmetic topology of 4-manifolds]{Arithmetic topology of 4-manifolds}

\author[Nikolaev]
{Igor V. Nikolaev$^1$}

\address{$^{1}$ Department of Mathematics and Computer Science, St.~John's University, 8000 Utopia Parkway,  
New York,  NY 11439, United States.}
\email{\textcolor[rgb]{0.00,0.00,0.84}{igor.v.nikolaev@gmail.com}}





\begin{abstract}
We construct a functor  from the  smooth 4-dimensional 
manifolds to  the hyper-algebraic number  fields, i.e. 
  fields with non-commutative multiplication. 
  It is proved  that  that the simply connected 4-manifolds correspond to 
  the  abelian   extensions.      
   We  recover the Rokhlin and  Donaldson's Theorems 
   from the Galois theory 
  of the non-commutative fields. 
  \end{abstract}

\maketitle

\section{Introduction}
 The arithmetic topology studies  an interplay   between  the 3-dimensional manifolds
 and the fields of algebraic numbers [Morishita 2012] \cite{M}.  
 The idea  dates back to C.~F.~Gauss. 
  Namely, there exists a map $F$ from   the $3$-dimensional manifolds $\mathscr{M}^3$ 
 to the algebraic  number fields $K$. 
 Such a map transforms   links  $\mathscr{Z}\subset\mathscr{M}^3$ (knots   $\mathscr{K}\subset \mathscr{M}^3$, resp.) into the ideals 
 (prime ideals, resp.) of  the ring of integers $O_K$ of  $K$.
The map  $F$ is a functor on the category of  3-dimensional 
 manifolds $\mathscr{M}^3$ with  values in the category of algebraic number fields $K$ \cite[Theorem 1.2]{Nik1}.

The aim of our note is an extension of $F$  to the smooth 4-dimensional manifolds $\mathscr{M}^4$.
The range of $F$ is  no longer the fields of algebraic numbers, but the  fields   $\mathbb{K}$ with
 non-commutative multiplication, e.g.   the quaternions or  the cyclic division algebras. 
 We refer to  $\mathbb{K}$ as a  hyper-algebraic number field.
  The Galois theory of  such fields was  elaborated 
by  [Cartan 1947] \cite{Car1} and  [Jacobson 1940]  \cite{Jac1},  see
{\cyr [Kharchenko} 1996] \cite{K} for a detailed account.

To introduce the map $F$, let  $N$ be a finite index  subgroup of the mapping class group $Mod ~\mathscr{M}^3$,
where   $\mathscr{M}^3$ is  a  3-dimensional manifold.  
The subgroup $N$ gives rise to  a smooth branched cover $\mathscr{M}^4_N$ of the 4-sphere $S^4$, 
see  [Piergallini 1995] \cite{Pie1}  and  Section 2.1. 
On the other hand, we have $\widehat{N} / \widehat{\mathbf{Z}} ~\cong  G_{\mathbb{K}}$,
where  $\widehat{N}$ ($\widehat{\mathbf{Z}}$, resp.) is a profinite completion of $N$ ($\mathbf{Z}$, resp.) 
and $G_{\mathbb{K}}$ is the absolute Galois group
of $\mathbb{K}$ \cite[Theorem 1.2]{Nik2}.
Recall  that  the  field $\mathbb{K}$ can be recovered up to an isomorphism from  the group $G_{\mathbb{K}}$, 
 see  [Uchida 1976] \cite[Corollary 2]{Uch1} and remark \ref{rmk2.4}. 
We define  $F$  as a composition of maps:
\begin{equation}\label{eq1.1}
\mathscr{M}^4_N\mapsto \widehat{N} / \widehat{\mathbf{Z}} ~\cong  G_{\mathbb{K}}\mapsto \mathbb{K}. 
\end{equation}

\bigskip
\begin{remark}\label{rmk1.1}
The map $F$ can be equivalently defined via a  $C^*$-algebra $\mathbb{E}_{\mathscr{M}^4}$ 
generated by  the diffeomorphisms of  $\mathscr{M}^4$ \cite{Nik4}. 
Namely,  the K-theory of $\mathbb{E}_{\mathscr{M}^4}$ gives rise to a number field $K$ whose 
central simple algebra is isomorphic to  $\mathbb{K}$.
In particular, the exotic smoothings of $\mathscr{M}^4$ are classified by the Brauer group 
of $K$, see  \cite{Nik4} for the details. 
 \end{remark}

\bigskip
Let $O_{\mathbb{K}}$ be the ring of integers of the field $\mathbb{K}$. 
Unlike  $O_K$,  the ring  $O_{\mathbb{K}}$ is always  simple, 
i.e. all two-sided ideals of  $O_{\mathbb{K}}$ are trivial. 
 Hence we must use the dynamical ideals (dynamials),
i.e. the crossed products $O_{\mathbb{K}}\rtimes m\mathbf{Z}$
by an endomprphism  of $O_{\mathbb{K}}$ of degree $m\ge 1$,  see  \cite{Nik3}
for the details and  motivation.  The dynamial   $O_{\mathbb{K}}\rtimes m\mathbf{Z}$ is minimal 
if and only if   $m=p$ is a prime number \cite[Theorem 1.7]{Nik3}.

To formalize our results, recall that 
the groups $N\not\cong N'$ are a Grothendieck pair, if $\widehat{N}\cong\widehat{N'}$. 
We shall denote by  $\mathfrak{M}^4$ a category of  all smooth  4-dimensional manifolds $\mathscr{M}^4_N$,
such that:  (i) the set $\mathfrak{M}^4$ contains no  $\mathscr{M}^4_N$ and  $\mathscr{M}^4_{N'}$ 
such that   $N$ and $N'$ are a Grothendieck pair
and (ii) the arrows of $\mathfrak{M}^4$ are  differentiable maps  between such  manifolds.
Denote by  $\mathfrak{K}$ a category of the hyper-algebraic number fields, such that
the arrows of $\mathfrak{K}$ are  injective homomorphisms between these fields.
The knotted surface   is a 
transverse immersion $\iota:  ~X_{g_1}\cup\dots X_{g_n}\hookrightarrow \mathscr{M}^4$ 
of a collection  of the 2-dimensional  orientable  surfaces  $X_{g_i}$ of the genera $g_i\ge 0$. 
The  $\iota(X_{g_1}\cup\dots X_{g_n})$ is called a surface knot $\mathscr{K}$  if $n=1$ and a  surface link 
$\mathscr{Z}$  otherwise. 
Our main result can be formulated  as follows. 
\begin{theorem}\label{thm1.1}
The map  $F: \mathfrak{M}^4\to\mathfrak{K}$ is a covariant functor,
such that: 

\medskip
(i) $F(S^4)\cong \mathbb{H}$ is the field of  quaternions; 

\smallskip
(ii) the dynamials  $O_{\mathbb{K}}\rtimes m\mathbf{Z}$ correspond to 
the surface links  $\mathscr{Z}\subset\mathscr{M}^4$; 

\smallskip
(iii)   the minimal dynamials $O_{\mathbb{K}}\rtimes p\mathbf{Z}$  correspond to
the surface knots  $\mathscr{K}\subset\mathscr{M}^4$. 
\end{theorem}
\begin{remark}
The functor $F$ is independent of  the choice of manifold $\mathscr{M}^3$ 
used in the construction of $F$  according to   the formula (\ref{eq1.1}). Indeed,  
using the manifolds 
$\left\{\mathscr{W}^4_i ~|~ \pi_1(\mathscr{W}^4_i)\cong Mod~\mathscr{M}^3_i\right\}$
described in Section 2.1, one can always take a connected sum $\# \mathscr{W}_i^4$,
so that $Mod~\mathscr{M}^3\cong \pi_1(\# \mathscr{W}_i^4 )$. 
This observation is corroborated by the $C^*$-algebra approach,  see remark \ref{rmk1.1}
 \end{remark}

Let $\mathbb{K}$ be  a Galois extension of $\mathbb{H}$,  i.e.  an  extension 
having the Galois group $Gal~\mathbb{K}$ and obeying  the fundamental 
correspondence of the Galois theory  [Cartan 1947] \cite{Car1},
[Jacobson 1940]  \cite{Jac1}  and {\cyr [Kharchenko} 1996] \cite{K}. 
Recall that  $\mathbb{K}$ is said to be an abelian extension, if $Gal~\mathbb{K}$ 
is an abelian group. Theorem \ref{thm1.1} implies the following result. 
\begin{corollary}\label{cor1.2}
Let $\mathscr{M}^4\in \mathfrak{M}^4$ be a  simply connected 4-manifold.
Then $\mathbb{K}=F(\mathscr{M}^4)$ is an abelian extension. 
 \end{corollary}
\begin{remark}\label{rmk1.3}
The converse of \ref{cor1.2} is false. 
\end{remark}
The article is organized as follows. Section 2 contains a brief review
of the 4-dimensional topology, the Galois theory of non-commutative fields  
and the arithmetic topology.  Theorem \ref{thm1.1}, corollary \ref{cor1.2}
and remark \ref{rmk1.3}   are proved 
in Section 3.  In Section 4 we give a proof of the Rokhlin and Donaldson's Theorems
using  the Galois theory of  non-commutative fields.

\section{Preliminaries}
This section contains  a brief review of the smooth 4-dimensional manifolds, 
the Galois theory of non-commutative  fields  
and the arithmetic topology.   We refer the reader to
 [Cartan 1947] \cite{Car1},  [Morishita 2012] \cite{M}
 and  [Piergallini 1995] \cite{Pie1} for a detailed account. 

\subsection{Topology of $4$-dimensional manifolds}
Let $\mathscr{M}^3$ be a closed 3-dimensional manifold. 
Denote by  $Mod ~\mathscr{M}^3$  the mapping class group of $\mathscr{M}^3$, i.e.
a group of isotopy classes of the orientation-preserving 
diffeomorphisms of $\mathscr{M}^3$.
The  $Mod ~\mathscr{M}^3$ is a finitely presented group
[Hatcher \& McCullough 1990] \cite{HaCu}.

Since each finitely presented
group can be realized as the fundamental group of a smooth 4-manifold,
we denote by $\mathscr{W}^4\in \mathfrak{M}^4$ a 4-manifold with
$\pi_1(\mathscr{W}^4)\cong Mod ~\mathscr{M}^3$.
It will be useful for us to represent  $\mathscr{W}^4$ as a PL (and, therefore, smooth) 
4-fold (or 5-fold, if necessary) branched  cover  of the sphere $S^4$. Let $X_g$ be a closed
surface, which we assume for the sake of brevity to be orientable 
of genus $g\ge 0$. Recall that there exists a transverse immersion 
$\iota: X_g\hookrightarrow S^4$, such that $\mathscr{W}^4$ is the 4-fold PL 
cover of  $S^4$ branched at the points of $X_g$ [Piergallini 1995] \cite{Pie1}. 
In other words, the $Mod ~\mathscr{M}^3$ is a normal subgroup of index 4
of the fundamental group $\pi_1(S^4-X_g)$. 
\begin{remark}\label{rmk2.1}
The immersion $\iota$ and surface $X_g$ need not be unique for given 
$\mathscr{W}^4$. Yet one can always restrict to a canonical choice of $\iota$ 
and $X_g$, which we always assume to be the case.
\end{remark}
\begin{definition}
By $\mathscr{M}^4_N$ we understand a
cover of the manifold $\mathscr{W}^4$  corresponding to 
the  finite index subgroup $N$ of the  fundamental group  
$\pi_1(\mathscr{W}^4)$. 
\end{definition}
\begin{remark}\label{rmk2.3}
In view of the continuous map $\mathscr{W}^4\to S^4$,
the manifold $\mathscr{M}^4_N$ is a smooth $4d$-fold branched cover of 
the sphere $S^4$, where $d$ is the index of $N$ in  $Mod ~\mathscr{M}^3$.
The manifold $\mathscr{M}^4_N$ is a regular (Galois) cover of $S^4$  if and only if $N$ 
is a normal subgroup of  $Mod ~\mathscr{M}^3$.
\end{remark}

\subsection{Galois theory for  non-commutative fields}
Denote by $\mathbb{K}$  a division ring, i.e. a ring such that
the set   $\mathbb{K}^{\times}:=\mathbb{K}-\{0\}$ is a group under  multiplication. 
If the group $\mathbb{K}^{\times}$ is commutative, then  $\mathbb{K}$ is a field. 
For otherwise, we refer to  $\mathbb{K}$ as a non-commutative field. 

Roughly speaking, the Galois theory for $\mathbb{K}$ is a correspondence between 
the subfields of $\mathbb{K}$ and the subgroups of a Galois group $Gal ~\mathbb{K}$
of the field $\mathbb{K}$.  Namely, let $G$ be a group of automorphisms of the field 
 $\mathbb{K}$.  It is easy to see, that the set 
 \begin{equation}
 \mathbb{I}_G=\{x\in \mathbb{K} ~|~ g(x)=x ~\hbox{for all} ~g\in G\}
\end{equation}
is a subfield of the field $\mathbb{K}$. 
\begin{definition}
An extension $\mathbb{K}$ of the field $\mathbb{L}$ is called Galois,
if there exists a group $G$ of automorphisms of  the field $\mathbb{K}$, such that 
$\mathbb{L}\cong\mathbb{I}_G$. The Galois group of $\mathbb{K}$ with respect
to $\mathbb{L}$ is defined as $Gal~\mathbb{K}\cong G$. 
The absolute Galois group $G_{\mathbb{K}}$ is defined as a Galois group of the
algebraic closure of $\mathbb{K}$.  The  $G_{\mathbb{K}}$ is a profinite group. 
\end{definition}
\begin{remark}\label{rmk2.4}
By an adaption of the argument for the case of fields [Uchida 1976] \cite[Corollary 2]{Uch1},
one can show that the group $G_{\mathbb{K}}$ defines the underlying  field
$\mathbb{K}$ up to an isomorphism, see lemma \ref{lm3.3}. 
\end{remark}

Unlike the case of fields, the  inner automorphisms of $\mathbb{K}$ 
are a non-trivial group. Indeed, consider an automorphism 
$h_g: \mathbb{K}\to \mathbb{K}$  given by the formula: 
 \begin{equation}\label{eq2.2}
x\mapsto g^{-1}xg, \quad x\in \mathbb{K}, \quad g\in \mathbb{K}^{\times}.
\end{equation}
 By $Inn~(\mathbb{K})$ we denote a group of such automorphisms 
under the composition $h_{g_1}\circ h_{g_2}=h_{g_1g_2}$,
where $g_1,g_2\in \mathbb{K}^{\times}$. 
It follows from (\ref{eq2.2}) that  $Inn~(\mathbb{K})\cong \mathbb{K}^{\times}/C$,
where $C$ is the center of  $\mathbb{K}^{\times}$.

 Let $G$ be a finite group of automorphisms of the field  $\mathbb{K}$.
A normal subgroup $\Gamma :=G ~\cap ~Inn~(\mathbb{K})$ of $G$ consists of 
the inner automorphisms of the field $\mathbb{K}$.  Consider a group
ring
 \begin{equation}\label{eq2.3}
\mathbb{B}(\Gamma)=\sum_{g_i\in\Gamma}\sum_{h_j\in C} g_i h_j.
\end{equation}
It is easy to see,  that  $\mathbb{B}(\Gamma)$ is a finite-dimensional algebra
over its center $C$. The following result has been established in 
 [Cartan 1947] \cite[Th\'eor\`eme 1]{Car1} and 
{\cyr [Kharchenko} 1996] \cite[p. 139]{K}. 
\begin{theorem}
If $\mathbb{K}$ is a Galois extension, then the corresponding 
Galois group $G$ satisfies  the short exact sequence of groups:
\begin{equation}
1\to\Gamma\buildrel\rm \iota\over\to G\to G/\Gamma\to 1,  
\end{equation}
where $\iota$ is an inclusion and $|G/\Gamma|=\dim_C \mathbb{B}(\Gamma)$. 
\end{theorem}
Recall that $\mathbb{B}(\Gamma)$ is a {\it Frobenius algebra} 
{\cyr [Kharchenko} 1996] \cite[Theorem 3.5.1]{K}.  
In other words, there exists a non-degenerate bilinear form
\begin{equation}\label{eq2.5}
Q: ~\mathbb{B}(\Gamma)\times \mathbb{B}(\Gamma)\to C,
\end{equation}
such that $Q(xy,z)=Q(x,yz)$ for all $x,y,z\in \mathbb{B}(\Gamma)$. 
\begin{remark}\label{rmk2.6}
The  form (\ref{eq2.5}) is symmetric if and only if $\mathbb{B}(\Gamma)$
is a commutative ring. 
 \end{remark}
\begin{proof}
$Q(xy,z)=Q(x,yz)=Q(x,zy)=Q(xz,y)=Q(zx,y)=Q(z,xy)$.
\end{proof}

\subsection{Arithmetic topology}
The arithmetic topology studies an interplay between
3-dimensional manifolds and  number fields 
 [Morishita 2012] \cite{M}.  
 Let $\mathfrak{M}^3$ be a category of  closed 3-dimensional manifolds,
such that  the arrows of $\mathfrak{M}^3$ are  homeomorphisms  between the  manifolds.
Likewise, let $\mathbf{K}$ be a category of the algebraic number fields,  where 
the arrows of $\mathbf{K}$ are  isomorphisms between such fields.
 Let  $\mathscr{M}^3\in \mathfrak{M}^3$ be a 3-manifold,  let $S^3\in \mathfrak{M}^3$ be the 3-sphere
 and let $O_K$ be the ring of integers of  $K\in\mathbf{K}$.  
\begin{theorem}\label{thm2.7}
{\bf (\cite[Theorem 1.2]{Nik1})}
The exists a covariant functor $F: \mathfrak{M}^3\to \mathbf{K}$, such that:

\medskip
(i) $F(S^3)=\mathbf{Q}$; 

\smallskip
(ii)  each  ideal $I\subseteq O_K=F(\mathscr{M}^3)$ corresponds to 
a link $\mathscr{Z}\subset\mathscr{M}^3$; 

\smallskip
(iii)  each  prime ideal $I\subseteq O_K=F(\mathscr{M}^3)$ corresponds to
a knot  $\mathscr{K}\subset\mathscr{M}^3$. 
\end{theorem}

\bigskip
The following construction of  $F$ is based on [Uchida 1976] \cite{Uch1}. 
 Let $X_{g,n}$ be an orientable surface of genus $g\ge 0$  with $n\ge 0$ boundary 
 components. Denote by $Mod ~X_{g,n}$ the mapping class group of
 $X_{g,n}$,  i.e. a group of isotopy classes of the orientation and boundary-preserving 
 diffeomorphisms of the surface $X_{g,n}$.  Let $N$ be a finite index  subgroup of  $Mod ~X_{g,n}$.
We omit the construction of a 3-manifold $\mathscr{M}^3_N\in \mathfrak{M}^3$ from  $N$ referring
 the reader to  \cite[Remark 1.1]{Nik1}.
As usual,  let $G_K$ be  the absolute Galois
group of  $K\in \mathbf{K}$, while $\widehat{N}$ and $\widehat{\mathbf{Z}}$ denote the profinite completion of 
the groups $N$ and $\mathbf{Z}$, respectively. 
Let $G_{K}\mapsto K$ be an injective map defined
in [Uchida 1976] \cite[Corollary 2]{Uch1}. 
\begin{theorem}\label{thm2.8}
{\bf (\cite[Theorem 1.2]{Nik2})}
The map:
\begin{equation}\label{eq2.6}
\mathscr{M}^3_N\mapsto \widehat{N} / \widehat{\mathbf{Z}} ~\cong  G_{K}\mapsto K 
\end{equation}
coincides with the functor $F$ of theorem \ref{thm2.7}. Such a  functor is injective, unless $N$ and $N'$ are
a Grothendieck pair. 
Moreover, for every normal  finite index subgroup $N'\subseteq N$,  there exists 
a  regular  cover $\mathscr{M}^3_{N'}$ of $\mathscr{M}^3_N$ and 
an intermediate  field $K'=F(\mathscr{M}^3_{N'})$, such that $K\subseteq K'$ and 
$Gal~(K'|K)\cong N/N'$. 
\end{theorem}
\begin{remark}
The map (\ref{eq2.6}) extends to  the 4-manifolds $\mathscr{M}^4$ and 
 non-commutative fields  $\mathbb{K}$, see (\ref{eq1.1}). 
Such an extension is at the heart  of our paper. 
\end{remark}

\section{Proofs}
\subsection{Proof of theorem \ref{thm1.1}}
For the sake brevity, we  shall focus on the first part by proving that  $F: \mathfrak{M}^4\to\mathfrak{K}$ is a covariant functor,
while referring the reader to \cite[Section 4]{Nik3} for the proof of items (i)-(iii) of theorem \ref{thm1.1}. 
We shall  prove $F$ to be a functor with respect to homeomorphisms first, and then generalize to arbitrary
differentiable mappings, see remark \ref{rmk3.4}.

\subsubsection{}
 Let us  show that $F: \mathfrak{M}^4\to\mathfrak{K}$ is a covariant functor. 
For clarity, we split the proof in a series of lemmas.
\begin{lemma}\label{lm3.1}
The manifolds $\mathscr{M}^4_N, \mathscr{M}^4_{N'}\in \mathfrak{M}^4$ 
are homeomorphic, if and only if, the subgroups $N, N'\subseteq Mod ~\mathscr{M}^3$ are isomorphic.
\end{lemma}
\begin{proof}
(i)  Let $\mathscr{M}^4_N$ be homeomorphic to  $\mathscr{M}^4_{N'}$
by a homeomorphism $h:  \mathscr{M}^4_N\to \mathscr{M}^4_{N'}$. 
Since both $\mathscr{M}^4_N$ and $\mathscr{M}^4_{N'}$ cover 
the 4-sphere $S^4$ branched over a surface $X_g\hookrightarrow S^4$,
one gets a commutative diagram in Figure 1. 
By definition, $N\cong p_*(\pi_1(\mathscr{M}^4_N))\subset \pi_1(S^4-X_g)$. 
Likewise, $N'\cong p_*'(\pi_1(\mathscr{M}^4_{N'}))\subset \pi_1(S^4-X_g)$. 
Since $\mathscr{M}^4_N$ is homeomorphic to $\mathscr{M}^4_{N'}$,
one gets an isomorphism of the fundamental groups 
$\pi_1(\mathscr{M}^4_N)\cong \pi_1(\mathscr{M}^4_{N'})$. 
Because the maps $p_*$ and $p_*'$ are injective, we conclude that
$N$ and  $N'$ are isomorphic subgroups of $\pi_1(S^4-X_g)$ and, therefore, of 
the group $Mod ~\mathscr{M}^3$. In view of the above and remark \ref{rmk2.1},
the necessary condition of  lemma \ref{lm3.1}
is proved. 

\begin{figure}
\begin{picture}(300,90)(-90,10)
\put(23,75){\vector(0,-1){35}}
\put(70,75){\vector(-1,-1){35}}
\put(40,83){\vector(1,0){30}}
\put(8,25){$S^4-X_g$}
\put(17,80){$\mathscr{M}^4_N$}
\put(75,80){$\mathscr{M}^4_{N'}$}
\put(50,90){$h$}
\put(10,55){$p$}
\put(65,55){$p'$}
\end{picture}
\caption{Branched covers of $S^4$.}
\end{figure}

\medskip
(ii)  Let $N\cong N'$ be isomorphic  subgroups of 
$Mod ~\mathscr{M}^3$. As explained in Section 2.1, the groups $N$ and $N'$ define a pair of 
branched  covers $\mathscr{M}^4_N$ and $\mathscr{M}^4_{N'}$ of $S^4$. 
In view of the diagram in Figure 1,   one obtains an inclusion of groups $N, N'\subset \pi_1(S^4-X_g)$.
Since $N\cong N'$, the subgroups $N$ and $N'$ are conjugate in the group $\pi_1(S^4-X_g)$.
In other words, there exists an element $\mathbf{g}\in  \pi_1(S^4-X_g)$, such that
$N'=\mathbf{g}^{-1}N\mathbf{g}$. It is well known, that the conjugate subgroups of 
$\pi_1(S^4-X_g)$ correspond to the homeomorphic branched covers $\mathscr{M}^4_N$ and $\mathscr{M}^4_{N'}$
 of the $S^4$.  Moreover, the homeomorphism $h: \mathscr{M}^4_N\to \mathscr{M}^4_{N'}$
 is realized by a deck transformation of the branched cover. 
 In view of the above and remark \ref{rmk2.1}, 
the sufficient condition of  lemma \ref{lm3.1}
is proved. 
\end{proof}

\begin{lemma}\label{lm2}
Up to the Grothendieck pairs, the subgroups $N$ and $N'$
of $Mod ~\mathscr{M}^3$
 are isomorphic, if and only if, 
the  groups  $\widehat{N} / \widehat{\mathbf{Z}}\cong G_{\mathbb{K}}$
and  $\widehat{N'} / \widehat{\mathbf{Z}}\cong G_{\mathbb{K}'}$ are isomorphic. 
\end{lemma}
\begin{proof}
(i) Let $N\cong N'$ be a pair of isomorphic subgroups of  $Mod ~\mathscr{M}^3$. 
Recall that a profinite group $\widehat{N}$ is a topological group defined  by the inverse limit
\begin{equation}\label{eq3.1}
\widehat{N}:=\varprojlim N/N_k,
\end{equation}
where $N_k$ runs through all  open normal finite index subgroups of 
$N$.   It follows from (\ref{eq3.1}), that  if $N\cong N'$, then $\widehat{N}\cong \widehat{N'}$. 
Since  $G_{\mathbb{K}}\cong \widehat{N} / \widehat{\mathbf{Z}}$
and  $G_{\mathbb{K}'}\cong \widehat{N'} / \widehat{\mathbf{Z}}$, we conclude that 
 $G_{\mathbb{K}}\cong G_{\mathbb{K}'}$.
The necessary condition of lemma \ref{lm2} is proved.

\medskip
(ii)  Suppose that  $G_{\mathbb{K}}\cong\widehat{N} / \widehat{\mathbf{Z}}$ and $G_{\mathbb{K}'}\cong\widehat{N'} / \widehat{\mathbf{Z}}$ 
are isomorphic groups.  Let us show, that $N\cong N'$ are isomorphic subgroups of $Mod ~\mathscr{M}^3$. 
To the contrary, let $N\not\cong N'$.  Since $N$ and $N'$ cannot be a Grothendieck pair,  we conclude that 
$\widehat{N}\not\cong\widehat{N'}$.
But  $G_{\mathbb{K}}\cong\widehat{N} / \widehat{\mathbf{Z}}$
and  $G_{\mathbb{K}'}\cong\widehat{N'} / \widehat{\mathbf{Z}}$ and, therefore, 
 $G_{\mathbb{K}}\not\cong G_{\mathbb{K}'}$.  One gets a contradiction
 proving the sufficient condition of lemma \ref{lm2}. 
\end{proof}

\begin{lemma}\label{lm3.3}
The absolute Galois groups  $G_{\mathbb{K}}$ and $G_{\mathbb{K}'}$
are isomorphic if and only if the the underlying non-commutative fields $\mathbb{K}$
and $\mathbb{K}'$ are isomorphic. 
\end{lemma}
\begin{proof}
The proof is an adaption of the argument of  [Uchida 1976] \cite[Corollary 2]{Uch1}
to the case of non-commutative fields. Namely, we introduce a topology on 
 $G_{\mathbb{K}}$  according to the formula (\ref{eq3.1}). 
 Let $G_{\mathbb{K}}$ and  $G_{\mathbb{K}'}$ be open subgroups of an absolute Galois group $G$,
 and let $\sigma: G_{\mathbb{K}}\to G_{\mathbb{K}'}$ be a topological isomorphism. 
 To prove lemma \ref{lm3.3},  it is enough to show that $\sigma$ 
 can be extended to an inner automorphism of   $G$,
 which corresponds   an isomorphism   $\mathbb{K}\cong\mathbb{K}'$. 
 Indeed, one takes an open normal subgroup $N$ of $G$  contained in 
  $G_{\mathbb{K}}$ and  $G_{\mathbb{K}'}$.  The group $N$ induces
  an isomorphism $\sigma_N: G_{\mathbb{K}}/N\to G_{\mathbb{K}'}/N$.
  It can be shown that $\sigma_N$ extends to an inner automorphism 
  of the group $G/N$.  We repeat the construction over all open normal subgroups
  of $G$ and obtain  an explicit formula for the required inner automorphism of $G$. 
Lemma \ref{lm3.3} follows. 
\end{proof}

\begin{remark}\label{rmk3.4}
Lemmas \ref{lm3.1}-\ref{lm3.3} are true for the differentiable mappings. 
In this case one obtains the inclusion of the corresponding non-commutative 
fields. 
\end{remark}

\bigskip
Lemmas \ref{lm3.1}-\ref{lm3.3} and formula 
(\ref{eq1.1}) imply that  $F: \mathfrak{M}^4\to\mathfrak{K}$ is a covariant functor.

\subsubsection{}
The detailed proof of  items (i)-(iii) of theorem \ref{thm1.1} is given in \cite[Theorem 4.1]{Nik3}. 
We refer the reader  to this paper for the proof and examples of the sphere knots.

\bigskip
Theorem \ref{thm1.1} is proved.

\subsection{Proof of corollary \ref{cor1.2}}
If  $\mathscr{M}^4\in\mathfrak{M}^4$ is  a simply connected manifold,
then $\pi_1(\mathscr{M}^4)$ is a trivial group. Since $\mathscr{M}^4$ 
is a cover of $S^4$ branched over a surface $X_g\hookrightarrow S^4$,
we conclude that $\pi_1(S^4-X_g)$ is a finite group. 
By the Hurewicz Theorem,  the  homology group 
\begin{equation}
H_1(S^4-X_g; \mathbf{Z})\cong 
\pi_1(S^4-X_g) ~/  ~[\pi_1(S^4-X_g), \pi_1(S^4-X_g)]
\end{equation}
 is also a finite abelian group.

 \medskip
 Consider a commutative diagram in Figure 2. 
 An automorphism of  $H_1(S^4-X_g; \mathbf{Z})$ comes from
 a homeomorphism $h: \mathscr{M}^4\to\mathscr{M}^4$ of the 
 manifold $\mathscr{M}^4$.  Theorem \ref{thm1.1} says that $h$ must correspond 
 to an automorphism $\sigma_h$ of the
 field $\mathbb{K}$, such that $\sigma_h(\mathbb{L})=
 \mathbb{L}$ for a subfield $\mathbb{L}\subseteq\mathbb{K}$. 
 (For simplicity, the reader can think $\mathbb{L}\cong\mathbb{H}$ is the field of quaternions.) 
 Since the automorphisms $\sigma_h$ generate the Galois 
 group of the extension $\mathbb{K} ~| ~\mathbb{L}$,  we conclude that:
\begin{equation}\label{eq3.3}
Gal~\mathbb{K}\cong Aut~(H_1(S^4-X_g; \mathbf{Z})). 
\end{equation}

\medskip
 Recall that the group of automorphisms of a  finite abelian group 
 is always an abelian group. Indeed, a finite abelian groups can be written
 in the form 
 $\mathbf{Z}/p_1\mathbf{Z}\oplus\dots\oplus \mathbf{Z}/p_k\mathbf{Z}$  for some distinct 
 primes $p_i$.  On the other hand, $Aut~(G\oplus H)\cong Aut~(G)\oplus Aut~(H)$,
 where $G$ and $H$ are finite abelian groups of the coprime order.  Thus
\begin{equation}\label{eq3.4}
Aut~(H_1(S^4-X_g; \mathbf{Z}))\cong Aut~( \mathbf{Z}/p_1\mathbf{Z})\oplus\dots\oplus Aut~(\mathbf{Z}/p_k\mathbf{Z}). 
\end{equation}

\medskip
Since the group of automorphism of a cyclic group is  a cyclic group,  using (\ref{eq3.4})  we conclude,  that
the  group  $Aut~(H_1(S^4-X_g; \mathbf{Z}))$ is abelian. In view of (\ref{eq3.3}), the group 
$Gal~\mathbb{K}$ is also abelian.  Corollary \ref{cor1.2} is proved.

\begin{figure}
\begin{picture}(300,90)(-90,10)
\put(10,40){\vector(0,1){35}}
\put(85,70){\vector(0,-1){30}}
\put(75,83){\vector(-1,0){30}}
\put(65,27){\vector(-1,0){18}}
\put(110,83){\vector(1,0){30}}
\put(110,27){\vector(1,0){30}}
\put(153,70){\vector(0,-1){30}}

\put(-25,25){$H_1(S^4-X_g;\mathbf{Z})$}
\put(-35,80){$H_1(\mathscr{M}^4;\mathbf{Z})\cong Id$}
\put(80,80){$\mathscr{M}^4$}
\put(120,90){$F$}
\put(120,35){$F$}
\put(70,25){$S^4-X_g$}
\put(150,80){$\mathbb{K}$}
\put(150,25){$\mathbb{L}$}
\end{picture}
\caption{Simply connected manifold $\mathscr{M}^4$.}
\end{figure}

\subsection{Proof of remark \ref{rmk1.3}}
Let  $\mathbb{K}$ be an abelian extension. Then $Gal ~\mathbb{K}$ is a 
finite abelian group.  Let us show,   that the group  $H_1(S^4-X_g; \mathbf{Z})$ 
can be infinite.  
Indeed,  one can always write $H_1(S^4-X_g; \mathbf{Z})\cong \mathbf{Z}^k\oplus Tors$,
where $k\ge 0$ and $Tors$ is a finite abelian group. 
As explained, we have  $Aut ~(H_1(S^4-X_g; \mathbf{Z}))
\cong Aut ~(\mathbf{Z}^k)\oplus Aut ~(Tors)$.  
Recall  that $Aut ~(\mathbf{Z}^k)\cong GL_k(\mathbf{Z})$.
If $k=1$,  then the group  $Aut ~(\mathbf{Z})\cong \mathbf{Z}/2\mathbf{Z}$
is  abelian and finite.  Yet the group  $H_1(S^4-X_g; \mathbf{Z})\cong \mathbf{Z}\oplus Tors$ 
is infinite.  In other words, the corresponding 4-manifold $\mathscr{M}^4\in\mathfrak{M}^4$
cannot be simply connected.


\section{Rokhlin and Donaldson's Theorems revisited}
In this section we give a proof of the Rokhlin and Donaldson's Theorems
based on  the Galois theory of  non-commutative fields. 
To outline the proof,  let $\mathscr{M}^4\in \mathfrak{M}^4$ be a  simply connected
smooth 4-manifold. The corollary \ref{cor1.2} says that $\mathbb{K}=F(\mathscr{M}^4)$ is an abelian
extension.  Consider a subgroup of the inner automorphisms $\Gamma$ of
the abelian group $Gal ~\mathbb{K}$.  Let  $\mathbb{B}(\Gamma)$ be the corresponding group ring, 
see  (\ref{eq2.3}).  Denote by $Q$  a symmetric bilinear form on  the $\mathbb{B}(\Gamma)$,
see   (\ref{eq2.5}) and remark \ref{rmk2.6}. For $\Gamma$  a finite  abelian group,
 we  calculate both $\mathbb{B}(\Gamma)$ and  $Q$, see
 lemmas \ref{lm4.1} and \ref{lm4.2}. 
 On the other hand, theorem \ref{thm1.1} implies an isomorphism of the $\mathbf{Z}$-modules:
\begin{equation}\label{eq4.1}
\mathbb{B}(\Gamma)\cong H_2(\mathscr{M}^4;\mathbf{Z}),
\end{equation}
 see lemma \ref{lm4.3}. 
 From the map $Q: \mathbb{B}(\Gamma)\times \mathbb{B}(\Gamma)\to\mathbf{Z}$, 
 one gets a  symmetric bilinear form on the homology group 
 $H_2(\mathscr{M}^4;\mathbf{Z})$. 
 As a corollary, we recover  from the arithmetic of $Q$ the  Rokhlin and Donaldson's Theorems
 for the simply connected smooth 4-manifolds.  
 Let us pass to a detailed argument.

 \bigskip
\begin{lemma}\label{lm4.1}
The group ring  of a finite abelian group $\Gamma\cong \mathbf{Z}/p_1\mathbf{Z}\oplus\dots\oplus \mathbf{Z}/p_k\mathbf{Z}$
is isomorphic to a direct sum of the cyclotomic fields, i.e.
\begin{equation}\label{eq4.2}
\mathbb{B}(\Gamma)\cong \mathbf{Z}(\zeta_{p_1})\oplus\dots\oplus\mathbf{Z}(\zeta_{p_k}), 
\end{equation}
where $\zeta_{p_i}$ is  the $p_i$-th root of unity. 
\end{lemma}
\begin{proof}
An elegant proof of this fact can be found in [Ayoub \& Ayoub 1969] \cite{AyAy1}. 
\end{proof}
 
 \bigskip
 Thus  to calculate  $Q$,  we can restrict to the cyclotomic fields $\mathbf{Q}(\zeta_{p_i})$ and 
 take  the  tensor product over  $p_i$.   Recall that the {\it trace form}  on $\mathbf{Q}(\zeta_{p_i})$
 is a symmetric bilinear form: 
\begin{equation}\label{eq4.3}
Tr_{\mathbf{Q}(\zeta_{p_i})}:  \mathbf{Q}(\zeta_{p_i})\times\mathbf{Q}(\zeta_{p_i})
\to\mathbf{Q},  \quad\hbox{such that}  \quad (x,y)\mapsto tr(xy),
\end{equation}
 where $tr$ is the trace of an algebraic number. The  trace form (\ref{eq4.3}) 
 is equivalent to  the form  $Tr_{\mathbf{Q}(\zeta_{p_i})}(x,x):= Tr_{\mathbf{Q}(\zeta_{p_i})}(x^2)$  
 via the formula:
\begin{equation}\label{eq4.4}
Tr_{\mathbf{Q}(\zeta_{p_i})}(x,y)=
{1\over 2}\left[ Tr_{\mathbf{Q}(\zeta_{p_i})}(x+y)^2)-Tr_{\mathbf{Q}(\zeta_{p_i})}(x^2)-Tr_{\mathbf{Q}(\zeta_{p_i})}(y^2)\right].
\end{equation}

\bigskip
\begin{lemma}\label{lm4.2}
The following classification  is true:
\begin{equation}\label{eq4.5}
Tr_{\mathbf{Q}(\zeta_{p})}(x^2)=
\left\{
\begin{array}{ccc}
p \langle 1\rangle, ~\hbox{if $p$ is odd prime}\\
2^n\langle 1\rangle\left(\langle 1\rangle\oplus \langle -1\rangle
\oplus (2^{n-1}-1)\times H\right),   ~\hbox{if $p=2^n, ~n\ge 4$}, 
\end{array}
\right.
\end{equation}
where we use the Witt ring notation $p \langle 1\rangle$ for the diagonal quadratic form of dimension $p$,
$2^n\langle 1\rangle$ for the Pfister form of dimension $2^n$ and $H$ for the hyperbolic (split) plane,
see  [Lam 2005] \cite[Chapter X]{L} for definitions.  
\end{lemma}
\begin{proof}
We refer the reader to  [Otake 2014] \cite[Theorem 2.1]{Ota1} 
 for the  proof. 
\end{proof}

\bigskip
\begin{remark}\label{rmk4.3}
The case $n=1$ ($n=2$; $n=3$) in formula (\ref{eq4.5}) corresponds to the complex numbers
(quaternions;  octonions), respectively; see  [Lam 2005] \cite[Chapter X]{L}. 
We omit these values of $n$,  since none of the fields  is contained in $\mathbb{K}$. 
\end{remark}

\bigskip
\begin{lemma}\label{lm4.3}
If  $H_2(\mathscr{M}^4;\mathbf{Z})$ is the second homology of a simply connected
manifold $\mathscr{M}^4$, then $\mathbb{B}(\Gamma)\cong H_2(\mathscr{M}^4;\mathbf{Z})$,
 where $\cong$ is an isomorphism of the corresponding $\mathbf{Z}$-modules. 
\end{lemma}
\begin{proof}
Recall that $Gal~\mathbb{K}$ is acting on the field $\mathbb{K}$ by the automorphisms
of $\mathbb{K}$. Theorem \ref{thm1.1} says that  such automorphisms correspond
to the homeomorphisms of the manifold $\mathscr{M}^4$. Recall that each homeomorphism
$h:  \mathscr{M}^4\to \mathscr{M}^4$ defines an linear map $h_*:  H_2(\mathscr{M}^4;\mathbf{Z})
\to  H_2(\mathscr{M}^4;\mathbf{Z})$. 
Since $\Gamma\subseteq Gal ~\mathbb{K}$, one gets a linear representation:
\begin{equation}\label{eq4.6}
\rho: \Gamma\to  Aut~(H_2(\mathscr{M}^4;\mathbf{Z})). 
\end{equation}

\medskip
Consider a regular representation of $\Gamma$ by the automorphisms of 
$H_2(\mathscr{M}^4;\mathbf{Z})\cong\mathbf{Z}^k$.  It is easy to see, that
such a representation coincides with $\rho$ and, therefore, we conclude that
$k=|\Gamma|$. Moreover, since $\mathbb{B}(\Gamma)$ is the group ring 
of $\Gamma$, one gets an isomorphism  $\mathbb{B}(\Gamma)\cong H_2(\mathscr{M}^4;\mathbf{Z})$
between the corresponding $\mathbf{Z}$-modules. Lemma \ref{lm4.3} is proved.
\end{proof}

\bigskip
\begin{corollary}\label{cor4.4}
{\bf (Rokhlin and Donaldson)}

\smallskip
(i)  Definite intersection form of a simply connected smooth 4-manifold  is diagonalizable;

\medskip
(ii) Signature of the intersection form of  a simply connected smooth 4-manifold   is divisible by 16. 

\end{corollary}
\begin{proof}
As explained, we  identify the intersection form: 
\begin{equation}\label{eq4.7}
 H_2(\mathscr{M}^4;\mathbf{Z})\times  H_2(\mathscr{M}^4;\mathbf{Z})\to\mathbf{Z}
\end{equation}
with  the trace form $Tr_{\mathbf{Q}(\zeta_{p})}(x,y)$  given  by the formulas (\ref{eq4.3})-(\ref{eq4.5}). 
Accordingly, we have to consider the following two cases.

\bigskip
(i)  Let us consider the case $Tr_{\mathbf{Q}(\zeta_{p})}(x^2)=p \langle 1\rangle$, where $p$ is an odd prime. 
Let $\{1,\zeta_p,\dots,\zeta_p^{p-1}\}$ be the standard basis in the ring of integers 
$\mathbf{Z}[\zeta_p]$ 
of the cyclotomic field $\mathbf{Q}(\zeta_p)$. The  corresponding formula (\ref{eq4.5})
can be written as:  
\begin{equation}\label{eq4.8}
Tr_{\mathbf{Q}(\zeta_{p})}(x^2)=\sum_{i=0}^{p-1}x_i^2,\qquad x_i\in\mathbf{Z}, 
\end{equation}
where $x=x_0+x_1\zeta_p+\dots+x_{p-1}\zeta_p^{p-1}$ for an $x\in\mathbf{Z}[\zeta_p]$. 
Using (\ref{eq4.4}),  one gets a symmetric bilinear form on $H_2(\mathscr{M}^4;\mathbf{Z})$:
\begin{equation}\label{eq4.9}
Tr_{\mathbf{Q}(\zeta_{p})}(x,y)=\sum_{i=0}^{p-1} x_iy_i, \qquad x_i,y_i\in\mathbf{Z}.
\end{equation}
It remains to notice, that $\mathscr{M}^4\in\mathfrak{M}^4$ is a smooth
manifold and (\ref{eq4.9}) is a positive definite diagonalizable intersection form.
Item (i) of corollary \ref{cor4.4} follows.

\bigskip
(ii)  Let us consider the case $Tr_{\mathbf{Q}(\zeta_{p})}(x^2)=2^n\langle 1\rangle\left(\langle 1\rangle\oplus \langle -1\rangle
\oplus (2^{n-1}-1)\times H\right)$, where $p=2^n$ and $n\ge 4$;  see remark  \ref{rmk4.3} explaining the restriction $n\ge 4$. 
Let $\{1,\zeta_{2^{n+1}},\dots,\zeta_{2^{n+1}}^{2^{n+1}-1}\}$ be the standard basis in the 
$\mathbf{Z}[\zeta_{2^{n+1}}]$.  The corresponding formula (\ref{eq4.5})
can be written as: 
\begin{equation}\label{eq4.10}
Tr_{\mathbf{Q}(\zeta_{2^{n+1}})}(x^2)=   
\sum_{i=0}^{2^n-1} x_i^2+\left(x_{2^n}^2-x_{2^n+1}^2\right)+\left(\sum_{i=2^n+2}^{3\times 2^{n-1}-1}x_i^2
-\sum_{i=3\times 2^{n-1}}^{2^{n+1}-1} x_i^2\right),
\end{equation}
where $x=x_0+\dots+x_{2^{n+1}-1} \zeta_{2^{n+1}}^{2^{n+1}-1}  \in \mathbf{Z}[\zeta_{2^{n+1}}]$. 
Using (\ref{eq4.4}),  one obtains a symmetric bilinear form:
\begin{eqnarray}\label{eq4.11}
Tr_{\mathbf{Q}(\zeta_{2^{n+1}})}(x,y) &=&   
\sum_{i=0}^{2^n-1} x_iy_i ~+\cr 
+\left(x_{2^n}y_{2^n}-x_{2^n+1}y_{2^n+1}\right)
&+&\left(\sum_{i=2^n+2}^{3\times 2^{n-1}-1}x_iy_i
-\sum_{i=3\times 2^{n-1}}^{2^{n+1}-1} x_iy_i\right). 
\end{eqnarray}
It is easy to see, that  $Tr_{\mathbf{Q}(\zeta_{2^{n+1}})}(x,y)$ is a diagonal bilinear form on 
$H_2(\mathscr{M}^4;\mathbf{Z})$.  The signature of    (\ref{eq4.11})  
is equal to $2^n$,  since  the number of positive and negative 1's for the terms in brackets is the same,
while the signature of the first sum is $2^n$. 
It remains to notice, that $n\ge 4$ and, therefore, the signature of $Tr_{\mathbf{Q}(\zeta_{2^{n+1}})}(x,y)$
is divisible by $16$. 
Since $\mathscr{M}^4$ is a simply connected smooth 4-manifold, one gets item (ii) of corollary \ref{cor4.4}. 
\end{proof}


\bibliographystyle{amsplain}


\end{document}